\documentclass{ifacconf}

\renewcommand\tilde{\widetilde}

\newtheorem{nnassumption}{\bf Assumption}

\newtheorem{nntheorem}{\bf Theorem}
\newenvironment{theorem}{\begin{nntheorem}\it}{\end{nntheorem}}
\newtheorem{nndefinition}{\bf Definition}

\newtheorem{nnproposition}{\bf Proposition}
\newenvironment{proposition}{\begin{nnproposition}\it}{\end{nnproposition}}
\newtheorem{nnproblem}{\bf Problem}

\newtheorem{nnlemma}{\bf Lemma}
\newenvironment{lemma}{\begin{nnlemma}\it}{\end{nnlemma}}
\newtheorem{nnremark}{\bf Remark}

\newenvironment{proof}{{\bf Proof.}}{\hfill \hspace*{1pt}\hfill $\bullet$}

\usepackage{amsfonts,latexsym,amsmath,amssymb}
\usepackage[mathscr]{eucal}
\usepackage{graphicx,color}
\usepackage{comment}

\def\sat{\mathrm{sat}}

\usepackage{graphicx}      
\usepackage{natbib}        
\begin{document}
\begin{frontmatter}

\title{Global stabilization of a Korteweg-de Vries equation with a distributed control saturated in $L^2$-norm} 

\thanks[footnoteinfo]{This works has been partially supported by Basal Project FB0008 AC3E.}

\author[First]{Swann Marx}, 
\author[Second]{Eduardo Cerpa}, 
\author[First]{Christophe Prieur},
\author[Third]{Vincent Andrieu}

\address[First]{Gipsa-lab, Department of Automatic Control, Grenoble Campus, 11 rue des Math\'ematiques, BP 46, 38402 Saint Martin d'H\`eres Cedex, {\tt\small E-mail: \{swann.marx,christophe.prieur\}@gipsa-lab.fr}.}
\address[Second]{Departamento de Matem\'atica, Universidad T\'ecnica Federico Santa Mar\'ia, Avda. Espa\~na 1680, Valpara\'iso, Chile, {\tt\small E-mail: eduardo.cerpa@usm.cl}.}
\address[Third]{Universit\'e Lyon 1 CNRS UMR 5007 LAGEP, France and Fachbereich C - Mathematik und Naturwissenschaften, Bergische Universit\"at Wuppertal, Gau\ss stra\ss e 20, 42097 Wuppertal, Germany, {\tt\small E-mail: vincent.andrieu@gmail.com}.}

\begin{abstract}                
This article deals with the design of saturated controls in the context of partial differential equations. It is focused on a Korteweg-de Vries equation, which is a nonlinear mathematical model of waves on shallow water surfaces. The aim of this article is to study the influence of a saturating in $L^2$-norm distributed control on the well-posedness and the stability of this equation. The well-posedness is proven applying a Banach fixed point theorem. The proof of the asymptotic stability of the closed-loop system  is tackled with a Lyapunov function together with a sector condition describing the saturating input. Some numerical simulations illustrate the stability of the closed-loop nonlinear partial differential equation.
\end{abstract}


\end{frontmatter}

\section{Introduction}
 
The Korteweg-de Vries equation (KdV for short)
\begin{equation}
y_t+y_{x}+y_{xxx}+yy_x=0,
\end{equation}
is a mathematical model of waves on shallow water surfaces. Its stabilizability properties have been deeply studied with no constraints on the control, as reviewed in 
 \cite{cerpa2013control, rosier-zhang}. In this article, we focus on the following controlled  KdV equation
\begin{equation}
\label{nlkdv}
\left\{
\begin{split}
&y_t+y_x+y_{xxx}+yy_x+f=0,\: (t,x)\in [0,+\infty)\times [0,L],\\
&y(t,0)=y(t,L)=y_x(t,L)=0,\: t\in[0,+\infty),\\
&y(0,x)=y_0(x),\: x\in [0,L],
\end{split}
\right.
\end{equation}
where $y$ stands for the state and $f$ for the control. As studied in \cite{rosier1997kdv}, if $f=0$ and
\begin{equation}
\label{critical-length}
L\in\left\{ 2\pi\sqrt{\frac{k^2+kl+l^2}{3}}\,\Big\slash \,k,l\in\mathbb{N}^*\right\},
\end{equation}
then, there exist solutions of the linearized version of \eqref{nlkdv} written as follows
\begin{equation}
\label{lkdv}
\left\{
\begin{split}
&y_t+y_x+y_{xxx}=0,\\
&y(t,0)=y(t,L)=0,\\
&y_x(t,L)=0,\\
&y(0,x)=y_0(x),
\end{split}
\right.
\end{equation}
for which the $L^2(0,L)$-energy does not decay to zero. For instance, if $L=2\pi$ and $y_0=1-\cos(x)$ for all $x\in [0,L]$, then $y(t,x)=1-\cos(x)$ is a stationary solution of \eqref{nlkdv} conserving the energy for any time $t$. 

Note however that, if $L=2\pi$ and $f=0$, the origin of \eqref{nlkdv} is locally asymptotically stable as stated in \cite{chu2013asymptotic}. Then the nonlinear version of the Korteweg-de Vries equation has better results in terms of stability than the linearized one. A similar phenomenon also appears when studying the controllability of the equation with a Neumann control (see e.g. \cite{coroncrepeau2004missed} or \cite{cerpa2007siam}).

In the literature there are some methods stabilizing the KdV equation \eqref{nlkdv} with boundary control, see \cite{cerpa2009rapid, cerpa_coron_backstepping, marx-cerpa} or distributed controls as studied in \cite{pazoto2005localizeddamping, rosier2006global}. 

In this paper, we deal with the case where the control is saturated. Indeed, in most applications, actuators are limited due to some physical constraints and the control input has to be bounded. Neglecting the amplitude actuator limitation can be source of undesirable and catastrophic behaviors for the closed-loop system. Nowadays, numerous techniques are available (see e.g. \cite{tarbouriech2011book_saturating,teel1992globalsaturation,sussmann1991saturation}) and such systems can be analyzed with an appropriate Lyapunov function and a sector condition of the saturation map, as introduced in \cite{tarbouriech2011book_saturating} or \cite{zaccarian2011modern}.

To the best of our knowledge, there are few papers studying this topic in the infinite dimensional case. Among them, there are \cite{lasiecka2002saturation}, \cite{prieur2016wavecone}, where a wave equation equipped with a saturated distributed actuator is considered and \cite{daafouz2014nonlinear}, where a coupled PDE/ODE system modeling a switched power converter with a transmission line is considered and, due to some restrictions on the system, a saturated feedback has to be designed. There exist also some papers using the nonlinear semigroup theory and focusing on abstract systems (\cite{Logemann98time-varyingand}, \cite{seidman2001note}, \cite{slemrod1989mcss}).
 
In \cite{mcpa2015kdv_saturating}, in which it is considered a linear Korteweg-de Vries equation with a saturated distributed control, nonlinear semigroup theory is applied. In the case of the present paper, since the term $yy_x$ is not globally Lipschitz, such a theory is harder to use. Thus, we aim at studying a particular nonlinear partial differential equation without seeing it as an abstract control system and without using the nonlinear semigroup theory. Moreover, in the case of the present paper, the saturation is borrowed from \cite{slemrod1989mcss} and allows us to give explicitely the decay rate for bounded initial conditions, although the saturation used in \cite{mcpa2015kdv_saturating} is the classical one and no decay rate is given. 

This article is organized as follows. In Section \ref{sec_mainresults}, we present our main results about the well posedness and the stability of \eqref{nlkdv} in presence of saturating control. Sections \ref{sec_wp} and \ref{sec_stab} are devoted to prove these results by using respectively a Banach fixed-point theorem and some Lyapunov techniques together with a sector condition describing the saturating input. In Section \ref{sec_simu}, we give some simulations of the equation looped by a saturated feedback law. Section \ref{sec_conc} collects some concluding remarks and possible further research lines. 

\textbf{Notation:} $y_t$ (resp. $y_x$) stands for the partial derivative of the function $y$ with respect to $t$ (resp. $x$). Given $L>0$, $\Vert \cdot\Vert_{L^2(0,L)}$ denotes the norm in $L^2(0,L)$ and  $H^1(0,L)$ (resp. $H^3(0,L)$) is the set of all functions $u\in L^2(0,L)$ such that $u_x\in L^2(0,L)$ (resp. $u_x, u_{xx}, u_{xxx}\in L^2(0,L)$).  A function $\alpha$ is said to be a class $\mathcal{K}$ function if it is a nonnegative, an increasing function and if $\alpha(0)=0$.



\section{Main results}
\label{sec_mainresults}
Setting a control $f(t,x)=ay(t,x)$ with $a>0$, we obtain that \eqref{nlkdv} is stabilized in the $L^2(0,L)$-norm topology. Indeed, performing some integrations by parts, we obtain, at least formally
\begin{equation}
\frac{1}{2}\frac{d}{dt}\int_0^L |y(t,x)|^2dx=-|y_x(t,0)|^2-a\int_0^L |y(t,x)|^2dx
\end{equation}
and thus
\begin{equation}
\label{rapid-stab}
\Vert y(t,.)\Vert_{L^2(0,L)}\leq e^{-at}\Vert y_0\Vert_{L^2(0,L)},\quad \forall t\geq 0.
\end{equation}
Let us consider the impact of the constraint on the control. For infinite-dimensional systems, a way to take into account this constraint is to use the following saturation function, that is for all $s\in L^2(0,L)$ and for all $x\in [0,L]$,
\begin{equation}
\label{function-saturation}
\sat(s)(x)=\left\{
\begin{array}{rl}
&\hspace{-1.3cm}s(x)\hspace{0.6cm}\text{ if }\Vert s\Vert_{L^2(0,L)}\leq u_s,\\
\frac{s(x)u_s}{\Vert s\Vert_{L^2(0,L)}}&\text{ if } \Vert s\Vert_{L^2(0,L)}\geq u_s.
\end{array}
\right.
\end{equation}
Let us consider the KdV equation controlled by a saturated distributed control as follows
\begin{equation}
\label{nlkdv_sat}
\left\{
\begin{split}
&y_t+y_x+y_{xxx}+yy_x+\sat(ay)=0,\\
&y(t,0)=y(t,L)=0,\\
&y_x(t,L)=0,\\
&y(0,x)=y_0(x).
\end{split}
\right.
\end{equation}
Let us state the main results of this paper.

\begin{theorem}[Well posedness]
\label{nl-theorem-wp}
For any initial conditions $y_0\in L^2(0,L)$, there exists a unique mild solution $y\in C(0,T;L^2(0,L))\cap L^2(0,T;H^1(0,L))$ to \eqref{nlkdv_sat}. 
\end{theorem}

\begin{theorem}[Global asymptotic stability]
\label{glob_as_stab}
Given positive values $a$ and $u_s$, there exists a class $\mathcal{K}$ function $\alpha:\mathbb{R}_{\geq 0}\rightarrow \mathbb{R}_{\geq 0}$ such that for a given $y_0\in L^2(0,L)$, the mild solution $y$ to (\ref{nlkdv_sat}) satisfies,
\begin{equation}
\Vert y(t,.)\Vert_{L^2(0,L)}\leq \alpha(\Vert y_0\Vert_{L^2(0,L)})e^{-at},\quad \forall t\geq 0.
\end{equation}
Moreover, for bounded initial conditions, we could estimate the decay rate of the solution. In other words, given a positive value $r$, for any initial condition $y_0\in L^2(0,L)$ such that $\Vert y_0\Vert_{L^2(0,L)}\leq r$, any mild solution $y$ to \eqref{nlkdv_sat} satisfies,
\begin{equation}
\label{local_as_stab_estim}
\Vert y(t,.)\Vert_{L^2(0,L)}\leq \Vert y_0\Vert_{L^2(0,L)}e^{-\mu t},\quad\forall t\geq 0,
\end{equation}
where $\mu$ is defined as follows
\begin{equation}
\label{decay-rate-local}
\mu:=\min\left\{a,\frac{u_s}{r}\right\}.
\end{equation} 
\end{theorem}
By a scaling, we may assume, without loss of generality, that either $a$ or $u_s$ is $1$. However, Theorem \ref{glob_as_stab} shows that saturating a controller insuring a rapid stabilization makes the origin globally asymptotically stable. We cannot select anymore the decay rate of the convergence.



\section{Proof of the main results}

\subsection{Well-posedness}
\label{sec_wp}

\subsubsection{Linear system.}

Before proving the well-posedness of \eqref{nlkdv_sat}, let us recall some useful results on the linear system \eqref{lkdv}. To do that, consider the operator defined by
$$
D(A)=\lbrace w\in H^3(0,L),\: w(0)=w(L)=w^\prime(L)=0\rbrace,
$$
$$
A:w\in D(A)\subset L^2(0,L)\longmapsto (-w^\prime-w^{\prime\prime\prime})\in L^2(0,L).
$$

It is easy to prove that $A$ generates a strongly continuous semigroup of contractions which we will denote by $W(t)$. We have the following theorem proven in \cite{rosier1997kdv}
\begin{theorem}[Well-posedness of \eqref{lkdv}, \cite{rosier1997kdv}]
\label{lkdv-wp}
For any initial conditions $y_0\in L^2(0,L)$, there exists a unique mild solution $y\in C(0,T;L^2(0,L))\cap L^2(0,T;H^1(0,L))$ to (\ref{lkdv}). Moreover, there exists $C>0$ such that the solution to (\ref{lkdv}) satisfies
\begin{equation}
\label{dissipativity-regularity}
\Vert y\Vert_{C(0,T;L^2(0,L))}+\Vert y\Vert_{L^2(0,T;H^1(0,L))}\leq C_1\Vert y_0\Vert_{L^2(0,L)}
\end{equation}
and the extra regularity
\begin{equation}
\label{extra-regularity}
\Vert y_x(.,0)\Vert_{L^2(0,T)}\leq \Vert y_0\Vert_{L^2(0,L)}.
\end{equation}
\end{theorem}

To ease the reading, let us denote the following Banach space, for all $T>0$, 
$$\mathcal{B}(T):=C(0,T;L^2(0,L))\cap L^2(0,T;H^1(0,L))$$
endowed with the norm 
\begin{equation}
\Vert y\Vert_{\mathcal{B}(T)}=\sup_{t\in [0,T]}\Vert y(t,.)\Vert_{L^2(0,L)}+\left(\int_0^T \Vert y(t,.)\Vert^2_{H^1(0,L)}dt\right)^{\frac{1}{2}}
\end{equation}

Before studying the well-posedness of (\ref{nlkdv_sat}), we need a well-posedness result with a right-hand side. Thus, given $g\in L^1(0,T;L^2(0,L)$, let us consider $y$ the unique solution \footnote{It follows from the semigroup theory the existence and the unicity of $y$ when $g\in L^1(0,T;L^2(0,L)$ (see \cite[Chapter 4]{pazy1983semigroups}).} to the following inhomogeneous problem:
\begin{equation}
\label{kdv-2}
\left\{
\begin{split}
&y_t+y_x+y_{xxx}=g,\\
&y(t,0)=y(t,L)=0,\\
&y_x(t,L)=0,\\
&y(0,.)=y_0.
\end{split}
\right.
\end{equation}

Note that we need the following property on the saturation function, which will allow us to state that this type of nonlinearity belongs to the space $L^1(0,T;L^2(0,L))$.
\begin{lemma}[\cite{slemrod1989mcss}, Theorem 5.1.]
\label{lipschitz-satl2}
For all $(s,\tilde{s})\in L^2(0,L)^2$, we have
\begin{equation}
\Vert \sat(s)-\sat(\tilde{s})\Vert_{L^2(0,L)}\leq 3\Vert s-\tilde{s}\Vert_{L^2(0,L)}.
\end{equation}
\end{lemma}

Let us now state the following properties on the nonhomogeneous linearized KdV equation \eqref{kdv-2}.

Firstly, we have this proposition borrowed from \cite[Proposition 4.1]{rosier1997kdv}
\begin{proposition}[\cite{rosier1997kdv}]
\label{proposition-reg-rosier}
If $y\in L^2(0,T;H^1(0,L))$, then $yy_x\in L^1(0,T;L^2(0,L))$ and the map $\psi_1: y\in L^2(0,T;H^1(0,L))\mapsto yy_x \in L^1(0,T;L^2(0,L))$ is continuous.
\end{proposition}
Secondly, we have the following proposition
\begin{proposition}
\label{proposition-reg}
If $y\in L^2(0,T;H^1(0,L))$, then  $\sat(ay)\in L^1(0,T;L^2(0,L))$ and the map $\psi_2: y\in L^2(0,T;H^1(0,L)) \mapsto \sat(ay)\in L^1(0,T;L^2(0,L))$ is continuous.
\end{proposition}
\begin{proof}
Let $y,z\in L^2(0,T;H^1(0,L))$. We have, using Lemma \ref{lipschitz-satl2} and H\"older inequality
\begin{eqnarray}
\label{regularity-sat-l1}
\Vert \sat(ay) - &&\sat(az)\Vert_{L^1(0,T;L^2(0,L))}, \nonumber \\
&& \leq 3\int_0^T a\Vert (y-z)\Vert_{L^2(0,L)},\nonumber \\
&&\leq 3 L a\sqrt{T}\Vert (y-z)\Vert_{L^2(0,T;H^1(0,L))}.
\end{eqnarray}
Plugging $z=0$ in (\ref{regularity-sat-l1}) yields $\sat(ay)\in L^1(0,T;L^2(0,L))$ and using (\ref{regularity-sat-l1}) gives the continuity of the map $\psi_2$. It concludes the proof of Proposition \ref{proposition-reg}.\end{proof}

Finally, we have this result borrowed from \cite[Proposition 4.1]{rosier1997kdv}
\begin{proposition}[\cite{rosier1997kdv}]
\label{rosier-fixed}
For $g\in L^1(0,T;L^2(0,L))$ and $y_0\in L^2(0,L)$, the mild solution of (\ref{kdv-2}) belongs to $\mathcal{B}(T)$. Moreover the map $\psi_3: g\in L^1(0,T;H^1(0,L))\mapsto y\in \mathcal{B}(T)$ is continuous. 

\end{proposition}

\subsubsection{Proof of Theorem \ref{nl-theorem-wp}}

We are now in position to prove Theorem \ref{nl-theorem-wp}. Let us begin this section with a technical lemma.
\begin{lemma}[\cite[Lemma 18]{chapouly2009global}]
\label{zhang-regularity}
For any $T>0$ and $y,z\in\mathcal{B}(T)$,
\begin{equation}
\int_0^T \Vert (y(t,.)z(t,.))_x\Vert_{L^2(0,L)}dt\leq 2 \sqrt{T}\Vert y\Vert_{\mathcal{B}(T)}\Vert z\Vert_{\mathcal{B}(T)}
\end{equation}
\end{lemma}

Let us now state our local well-posedness result.
\begin{lemma}[\em Local well-posedness of (\ref{nlkdv_sat})\em]
\label{local-wp}
Let $T>0$, be given. For any $y_0\in L^2(0,L)$, there exists $T^\prime \in [0,T]$ depending on $y_0$ such that \eqref{nlkdv_sat} admits a unique solution $y\in\mathcal{B}(T^\prime)$. 
\end{lemma}
\begin{proof}
We follow the strategy of \cite{chapouly2009global} and \cite{rosier2006global}. From Propositions \ref{proposition-reg-rosier}, \ref{proposition-reg} and \ref{rosier-fixed}, we know that, for all $z\in L^2(0,T;H^1(0,L))$, there exists a unique solution $y$ in $\mathcal{B}(T)$ to the following equation
\begin{equation}
\label{kdv-fixed}
\left\{
\begin{split}
&y_t+y_x+y_{xxx}=-zz_x-\sat(az),\\
&y(t,0)=y(t,L)=0,\\
&y_x(t,L)=0,\\
&y(0,x)=y_0(x).
\end{split}
\right.
\end{equation}
Solution $y$ to \eqref{kdv-fixed} can be written in its integral form
\begin{equation}
\begin{split}
y(t)=&W(t)y_0-\int_0^t W(t-\tau)(zz_x)(\tau)d\tau\\
&-\int_0^t W(t-\tau)\sat(az(\tau,.))d\tau
\end{split}
\end{equation}
For given $y_0\in L^2(0,L)$, let $r$ and $T^\prime$ be positive values to be chosen. Let us consider the following set
\begin{equation}
S_{T^\prime ,r}=\lbrace z\in \mathcal{B}(T^\prime),\: \Vert z\Vert_{\mathcal{B}(T^\prime)}\leq r\rbrace,
\end{equation}
which is a closed, convex and bounded subset of $\mathcal{B}(T^\prime)$. Consequently, $S_{T^\prime , r}$ is a complete metric space in the topology induced from $\mathcal{B}(T)$. We define the map $\Gamma$ on $S_{T^\prime,r}$ by, for all $z\in S_{T^\prime,r}$
\begin{equation}
\begin{split}
\Gamma(z):=&W(t)y_0-\int_0^t W(t-\tau)(zz_x)(\tau)d\tau\\
&-\int_0^t W(t-\tau)\sat(az(\tau,.))d\tau,\: \forall z\in S_{T^\prime,r}.
\end{split}
\end{equation}
We aim at proving the existence of a fixed point for this operator.

It follows immediatly from \eqref{regularity-sat-l1}, Lemma \ref{zhang-regularity} and the linear estimates written in Theorem \ref{lkdv-wp} that, for every $z\in S_{T^\prime,r}$, there exist positive values $C_2$ and $C_3$ such that
\begin{equation}
\begin{split}
\Vert \Gamma(z)\Vert_{\mathcal{B}(T^\prime)}\leq &C_1 \Vert y_0\Vert_{L^2(0,L)}+C_2 \sqrt{T^{\prime}}\Vert z\Vert_{\mathcal{B}(T^\prime)}^2\\
&+C_3\sqrt{T^{\prime}}\Vert z\Vert_{\mathcal{B}(T^\prime)}
\end{split}
\end{equation}
We choose $r>0$ and $T^\prime>0$ such that 
\begin{equation}
\left\{
\begin{split}
&r=2C_1\Vert y_0\Vert_{L^2(0,L)}\\
&C_2\sqrt{T^\prime}r + C_3\sqrt{T^\prime}\leq \frac{1}{2}
\end{split}
\right.
\end{equation}
in order to obtain
\begin{equation}
\Vert \Gamma(z)\Vert_{\mathcal{B}(T^\prime)}\leq r,\quad \forall z\in S_{T^\prime,r}.
\end{equation}
Thus, with such $r$ and $T^\prime$, $\Gamma$ maps $S_{T^\prime ,r}$ to $S_{T^\prime ,r}$. Moreover, one can prove with the same inequalities that
\begin{equation}
\Vert \Gamma(z_1)-\Gamma(z_2)\Vert_{\mathcal{B}(T^\prime)}\leq \frac{1}{2}\Vert z_1-z_2\Vert_{\mathcal{B}(T^\prime)},\: \forall (z_1,z_2)\in S^2_{T^\prime ,r}
\end{equation}
The existence of the solutions to the Cauchy problem \eqref{nlkdv_sat} follows by using the Banach fixed point theorem \cite[Theorem 5.7]{brezis2010functional}.\end{proof}

We need the following Lemma inspired by \cite{coroncrepeau2004missed} and \cite{chapouly2009global} which implies that if there exists a solution for all $T>0$ then the solution is unique.
\begin{lemma}
\label{uniqueness-coron-crepeau}
For any $T>0$ and $a>0$, there exists $C_4(T,L)$ such that for every $y_0,z_0\in L^2(0,L)$ for which there exist mild solutions $y$ and $z$ to
\begin{equation}
\left\{
\begin{split}
&y_t+y_x+y_{xxx}+yy_x+\sat(ay)=0,\\
&y(t,0)=y(t,L)=0,\\
&y_x(t,L)=0,\\
&y(0,x)=y_0(x),
\end{split}
\right.
\end{equation}
and
\begin{equation}
\left\{
\begin{split}
&z_t+z_x+z_{xxx}+zz_x+\sat(az)=0,\\
&z(t,0)=z(t,L)=0,\\
&z_x(t,L)=0,\\
&z(0,x)=z_0(x),
\end{split}
\right.
\end{equation}
one has the following inequalities
\begin{equation}
\begin{split}
&\int_0^T \int_0^L (z_x(t,x)-y_x(t,x))^2dxdt\\
& \leq e^{C_4(1+\Vert y\Vert_{L^2(0,T;H^1(0,L))}+\Vert z\Vert_{L^2(0,T;H^1(0,L))})}\int_0^L (z_0-y_0)^2dx,
\end{split}
\end{equation} 
\begin{equation}
\begin{split}
&\int_0^T \int_0^L (z(t,x)-y(t,x))^2dxdt\\
&\leq e^{C_4(1+\Vert y\Vert_{L^2(0,T;H^1(0,L))}+\Vert z\Vert_{L^2(0,T;H^1(0,L))})}\int_0^L (z_0-y_0)^2dx.
\end{split}
\end{equation}
\end{lemma} 

\begin{proof}
Due to space limitation the proof has not been written. The interested reader can however refer to \cite{coroncrepeau2004missed} and \cite{chapouly2009global} where proofs of very similar results are provided.
\end{proof}

We aim at removing the smallness condition given by $T^\prime$ in Lemma \ref{local-wp}. Since we have the local well-posedness, we only need to prove the following a priori estimate for any solution to (\ref{nlkdv_sat}). 

\begin{lemma}
\label{global-estimation}
For given $T>0$, there exists $K:=K(T)>0$ such that for any $y_0\in L^2(0,L)$, for any $0< T^\prime\leq T$ and for any mild solution $y\in \mathcal{B}(T^\prime)$ to (\ref{nlkdv_sat}), it holds
\begin{equation}
\label{glob-estim-chapouly}
\Vert y\Vert_{\mathcal{B}(T^\prime)}\leq K\Vert y_0\Vert_{L^2(0,L)}.
\end{equation}
\end{lemma}

\begin{proof}
Let us fix $0<T^\prime\leq T$. Let us multiply the first line of (\ref{nlkdv_sat}) by $y$ and integrate on $(0,L)$. Using the boundary conditions in \eqref{nlkdv_sat}, we get the following estimates
\begin{equation*}
\int_0^L yy_xdx=0,\hspace{0.3cm}\int_0^L yy_{xxx}dx=\frac{1}{2}|y_x(t,0)|^2,
\end{equation*}
\begin{equation*}
\int_0^L y^2y_xdx=0.
\end{equation*}
Using the fact that $\sat$ is odd, we get that
\begin{equation}
\label{l2-dissipativity}
\frac{1}{2}\frac{d}{dt}\Vert y(t,.)\Vert^2_{L^2(0,L)} \leq -\frac{1}{2}|y_x(t,0)|^2-\int_0^L y\sat(ay)dx\leq 0
\end{equation}
and consequently
\begin{equation}
\label{global-wp-1}
\Vert y\Vert_{L^\infty(0,T^\prime;L^2(0,L)}\leq \Vert y_0\Vert_{L^2(0,L)}.
\end{equation}
It remains to prove a similar inequality for $$\Vert y_x\Vert_{L^2(0,T^\prime;L^2(0,L))},$$
to achieve the proof. To do that, we multiply by $xy$ (\ref{nlkdv_sat}), integrate on $(0,L)$ and use the following
\begin{equation*}
\int_0^L xyy_xdx=-\frac{1}{2}\Vert y\Vert^2_{L^2(0,L)},\hspace{0.2cm} \int_0^L yy_{xxx}dx=\frac{3}{2}\Vert y_x\Vert^2_{L^2(0,L)},
\end{equation*}
\begin{equation}
\label{nl-kdv-diss}
\begin{split}
-\int_0^L xy^2y_xdx &=\left|\frac{1}{3}\int_0^L y^3(t,x)dx\right|,\\
&\leq \frac{1}{3}\sup_{x\in [0,L]}|y(t,x)|\Vert y\Vert^2_{L^\infty(0,T^\prime;L^2(0,L))},\\
&\leq \frac{\sqrt{L}}{3}\Vert y_x\Vert_{L^2(0,L)}\Vert y\Vert^2_{L^\infty(0,T^\prime;L^2(0,L))},\\
&\leq \frac{\sqrt{L}\delta}{6}\Vert y_x\Vert_{L^2(0,L)}+\frac{\sqrt{L}}{6\delta}\Vert y\Vert^4_{L^\infty(0,T^\prime;L^2(0,L))},
\end{split}
\end{equation}
where $\delta$ is chosen such that $\delta:=\frac{3}{\sqrt{L}}$. In this way, we have
\begin{equation}
\begin{split}
\frac{1}{2}\frac{d}{dt}\int_0^L |x^{1/2}y(t,.)|^2&dx-\frac{1}{2} \int_0^L y^2dx+\frac{3}{2}\int_0^L |y_x|^2dx\\
&-\int_0^L xy^2y_xdx =-\int_0^L x\sat(ay)ydx.
\end{split}
\end{equation}
We get using \eqref{nl-kdv-diss} and the fact that $\sat$ is odd
\begin{equation}
\label{H1-inequality}
\begin{split}
\frac{1}{2}\frac{d}{dt}&\int_0^L |x^{1/2}y(t,.)|^2dx +\int_0^L |y_x|^2dx \leq  \\
&\frac{1}{2}\Vert y \Vert^2_{L^\infty(0,T^\prime;L^2(0,L))}+\frac{L}{18}\Vert y\Vert^4_{L^\infty(0,T^\prime;L^2(0,L))}.
\end{split}
\end{equation}
Using (\ref{global-wp-1}) and a Gr\"onwall inequality, we get the existence of a positive value $C_5=C_5(L)>0$ such that
\begin{equation}
\Vert y_x\Vert_{L^2(0,T^\prime;L^2(0,L))}\leq C_5\Vert y_0\Vert_{L^2(0,L)},
\end{equation}
which concludes the proof of Lemma \ref{global-estimation}.
\end{proof}

Using Lemmas \ref{local-wp}, \ref{uniqueness-coron-crepeau} and \ref{global-estimation}, for any $T>0$, we can conclude that there exists a unique mild solution in $\mathcal{B}(T)$ to (\ref{nlkdv_sat}). Indeed, with Lemma \ref{local-wp}, we know that there exists $T^\prime\in (0,T)$ such that there exists a unique solution to \eqref{nlkdv_sat} in $\mathcal{B}(T^\prime)$. Lemma \ref{uniqueness-coron-crepeau} states that if there exists a solution to $\mathcal{B}(T)$, then it is unique. Finally, Lemma \ref{global-estimation} allows us to state the well-posedness for every $T>0$: since the solution $y$ to \eqref{nlkdv_sat} is bounded by its initial condition for every $T^\prime>0$ belonging to $[0,T]$ as stated in \eqref{l2-dissipativity}, we know that there exists an unique solution to \eqref{nlkdv_sat} in $\mathcal{B}(T)$. This concludes the proof of Theorem \ref{nl-theorem-wp}. \hfill \hspace*{1pt}\hfill $\Box$

\subsection{Stability}
\label{sec_stab}
This section is devoted to the proof of Theorem \ref{glob_as_stab}. We need then to prove this lemma.
\begin{lemma}
\label{local_as_stab}
System (\ref{nlkdv_sat}) is \em semi-globally exponentially stable \em in $L^2(0,L)$. In other words, for any $r>0$ there exists a positive constant $\mu$ given in Theorem \ref{glob_as_stab} such that for any $y_0\in L^2(0,L)$ satisfying $\Vert y_0\Vert_{L^2(0,L)}\leq r$, the mild solution $y=y(t,x)$ to (\ref{nlkdv_sat}) satisfies
\begin{equation}
\label{local_exp_def}
\Vert y(t,.)\Vert_{L^2(0,L)}\leq \Vert y_0\Vert_{L^2(0,L)}e^{-\mu t}\qquad \forall t\geq 0.
\end{equation}
\end{lemma}
From this result, we will be able to prove the global asymptotic stability of \eqref{nlkdv_sat}, as done at the end of this section. Note moreover that this result is indeed the second part of Theorem \ref{glob_as_stab}. 

\subsubsection{Technical lemma.}
Before starting the proof of the Lemma \ref{local_as_stab}, let us state and prove the following lemma
\begin{lemma}[Sector condition]
Let $r$ be a positive value.
\label{sat-l2-local}
Given a positive value $a$ and $s\in L^2(0,L)$ such that $\Vert s\Vert_{L^2(0,L)}\leq r$, we have
\begin{equation}
\beta(x):=\Big(\sat(as)(x)-k(r) as(x)\Big)s(x)\geq 0,\quad \forall x\in [0,L],
\end{equation}
with
\begin{equation}
\label{sec-cond-gain}
k(r)=\min\left\{\frac{u_s}{ar},1 \right\}.
\end{equation}
\end{lemma}

\begin{proof}
 Consider the two following cases
\begin{itemize}
\item[1.] $\Vert as\Vert_{L^2(0,L)}\geq u_s$;
\item[2.] $\Vert as\Vert_{L^2(0,L)}\leq u_s$.
\end{itemize}
The first case implies that, for all $x\in [0,L]$
$$
\sat(as)(x)=\frac{as(x)}{\Vert as\Vert_{L^2(0,L)}}u_s.
$$
Thus, for all $x\in [0,L]$
\begin{equation}
\beta(x)=as(x)^2\left(\frac{u_s}{\Vert as\Vert_{L^2(0,L)}}-k(r)\right).
\end{equation}
Since we have
$$
\frac{u_s}{\Vert as\Vert_{L^2(0,L)}} \geq  \frac{u_s}{a r} \geq k(r),
$$
then, we obtain
$$
\beta(x)\geq 0.
$$

The second case implies that, for all $x\in [0,L]$
$$
\sat(as)(x)=as(x)
$$
We obtain that
$$
\beta(x):=(1-k(r))as(x)^2\geq 0.
$$

It concludes the proof of Lemma \ref{sat-l2-local}.
\end{proof}

\subsubsection{Semi-global exponential stability.}
Now we are able to prove Lemma \ref{local_as_stab}. Let $r$ be a positive value such that $\Vert y_0\Vert_{L^2(0,L)}\leq r$.

Note that from \eqref{l2-dissipativity}, we get
\begin{equation}
\label{dissipativity-nl}
\begin{split}
\Vert y\Vert_{L^2(0,L)} &\leq \Vert y_0\Vert_{L^2(0,L)}\\
&\leq r
\end{split}
\end{equation}
With the Lemma \ref{sat-l2-local} and \eqref{l2-dissipativity}, we obtain
\begin{equation}
\frac{1}{2}\frac{d}{dt}\int_0^L |y(t,x)|^2dx\leq -\int_0^L k(r)a |y(t,x)|^2dx
\end{equation}

Thus, with \eqref{sec-cond-gain}, applying the Gr\"onwall lemma leads to
\begin{equation}
\Vert y(t,.)\Vert_{L^2(0,L)}\leq e^{-\mu t}\Vert y_0\Vert_{L^2(0,L)}
\end{equation}
where $\mu$ is defined in the statement of Theorem \ref{glob_as_stab}. Therefore, it concludes the proof of Lemma \ref{local_as_stab}. \null\hfill$\Box$ 

\subsubsection{Proof of Theorem \ref{glob_as_stab}}

We are now in position to prove Theorem \ref{glob_as_stab}, inspired by \cite{rosier2006global}.
 
\begin{proof}
By Lemma \ref{local_as_stab} and \eqref{l2-dissipativity}, if
\begin{equation}
\Vert \tilde{y}_0\Vert_{L^2(0,L)}\leq \frac{u_s}{a},
\end{equation}
then, by dissipativity, we have $\Vert \tilde{y}(t,.)\Vert_{L^2(0,L)}\leq \frac{u_s}{a}$ for all $t\geq 0$. Thus, $\sat(a\tilde{y})=a\tilde{y}$ and we know, from \eqref{rapid-stab}, that the corresponding solution $\tilde{y}$ to (\ref{nlkdv_sat}) satisfies
\begin{equation}
\label{1-nl-stab-lyap}
\Vert \tilde{y}(t,.)\Vert_{L^2(0,L)}\leq \Vert \tilde{y}_0\Vert_{L^2(0,L)}e^{-at}\qquad \forall t\geq 0.
\end{equation}
In addition, for a given $r>0$, there exists a positive constant $\mu$, which is given in \eqref{decay-rate-local}, such that if $\Vert y_0\Vert_{L^2(0,L)}\leq r$, then any mild solution $y$ to (\ref{nlkdv_sat}) satisfies
\begin{equation}
\Vert y(t,.)\Vert_{L^2(0,L)}\leq \Vert y_0\Vert_{L^2(0,L)}e^{-\mu t}\qquad \forall t\geq 0
\end{equation}
Consequently, setting $T_r:=\mu^{-1}\ln\left(\frac{ar}{u_s}\right)$, we have
$$
\Vert y_0\Vert_{L^2(0,L)}\leq r\Rightarrow \Vert y(T_r,.)\Vert_{L^2(0,L)}\leq e^{-\ln\left(\frac{ar}{u_s}\right)}r=\frac{u_s}{a}
$$
Therefore, using \eqref{1-nl-stab-lyap}, we obtain
\begin{equation*}
\begin{split}
\Vert y(t,.)\Vert_{L^2(0,L)} &\leq \Vert y(T_r,.)\Vert_{L^2(0,L)}e^{-a(t-T_r)},\qquad \forall t\geq T_r\\
&\leq \Vert y_0\Vert_{L^2(0,L)}e^{aT_r}e^{-at},\qquad \forall t\geq 0.
\end{split}
\end{equation*}
Thus it concludes the proof of Theorem \ref{glob_as_stab}.
\end{proof}

\section{Simulation}
\label{sec_simu}

Let us discretize the PDE \eqref{nlkdv_sat} by means of finite difference method (see e.g. \cite{nm_KdV} for an introduction on the numerical scheme of a generalized Korteweg-de Vries equation). The time and the space steps are chosen such that the stability condition of the numerical scheme is satisfied (see \cite{nm_KdV} where this stability condition is clearly established).

Given $y_0(x)=1-\cos(x)$, $L=2\pi$, $a=1$ and $T_{final}=6$, Figure \ref{figure1} shows the solution to \eqref{nlkdv_sat} and with the unsaturated control $f=ay$. Figure \ref{figure2} illustrates the simulated solution with the same initial condition and a saturated control $f=\sat(ay)$ where $u_s=0.5$.
The evolution of the $L^2$-energy of the solution in these two cases is given by Figure \ref{figure4}.
\begin{figure}[h!]
   \begin{minipage}[c]{.46\linewidth}
      \includegraphics[scale=0.2]{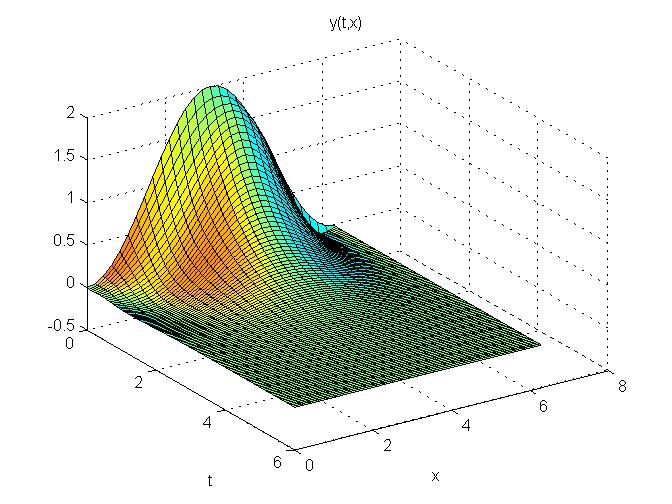}
      \caption{Solution $y(t,x)$ with a feedback law $f=ay$.}
      \label{figure1}
   \end{minipage} \hfill
   \begin{minipage}[c]{.46\linewidth}
      \includegraphics[scale=0.2]{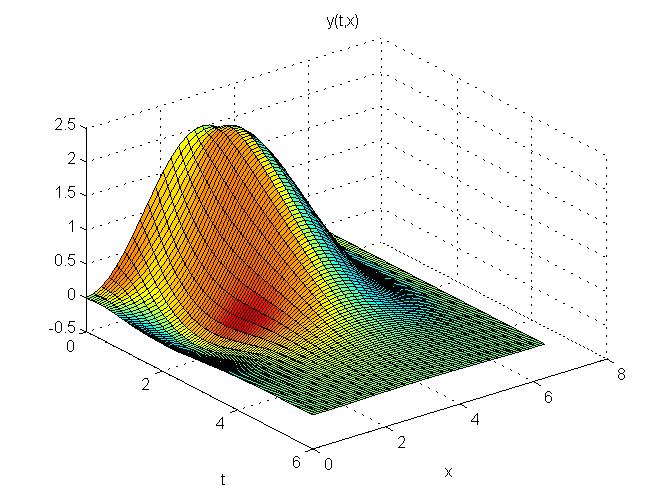}
      \caption{Solution $y(t,x)$ with a feedback $f=\sat(ay)$ where $u_s=0.5$.}
      \label{figure2}
   \end{minipage}\hfill
\end{figure}
\begin{figure}[h!]
      \includegraphics[scale=0.3]{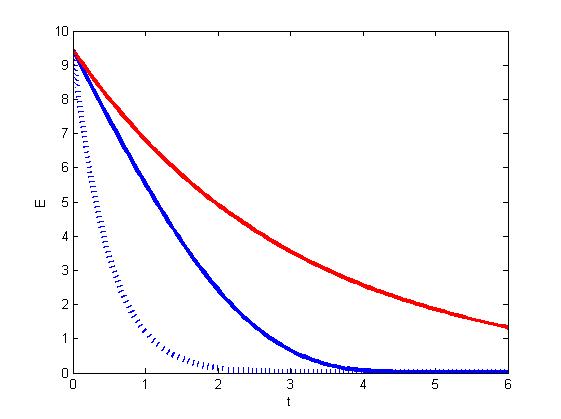}
      \caption{Blue: Time evolution of the energy function $\Vert y\Vert^2_{L^2(0,L)}$ with a saturation $u_s=0.5$ and $a=1$. Red: Time evolution of the theoritical energy $\Vert y_0\Vert^2_{L^2(0,L)}e^{-2\mu t}$. Dotted line: Time evolution of the solution without saturation and $a=1$.}
      \label{figure4}
\end{figure}

\section{Conclusion}
\label{sec_conc}

In this paper, we have studied the well-posedness and the asymptotic stability of a Korteweg-de Vries equation with a saturated distributed control. The well-posedness issue has been tackled by using the Banach fixed-point theorem and we proved the stability by using a sector condition and Lyapunov theory for infinite dimensional systems. Two questions may arise in this context. In \cite{mcpa2015kdv_saturating}, an other saturation function is used. Is the equation with a control saturated with such a function still globally asymptotically stable? Some boundary controls have been already designed in \cite{cerpa_coron_backstepping}, \cite{coron2014local}, \cite{tang2013stabKdV} or \cite{cerpa2009rapid}. By saturating this controllers, are the corresponding equations still stable? 



\bibliography{bibsm}

\end{document}